\documentclass[12pt]{amsart}
\usepackage[mathcal]{eucal}
\usepackage{amssymb, graphicx, labelfig}

\newtheorem{thm}{Theorem}
\newtheorem{prop}[thm]{Proposition}
\newtheorem{lem}[thm]{Lemma}

\theoremstyle{remark}

\theoremstyle{definition}

\newcommand{\col}{\kern -3pt :}
\newcommand{\C}{\mathbb C}
\newcommand{\R}{\mathbb R}

\newcommand{\HH}{\mathbb H}
\newcommand{\Id}{\mathrm{Id}}
\newcommand{\End}{\mathrm{End}}
\newcommand{\Isom}{\mathrm{Isom}^+}
\newcommand{\Aut}{\mathrm{Aut}}
\newcommand{\Diff}{\mathrm{Diff}}
\newcommand{\Ad}{\mathrm{Ad}}
\newcommand{\T}{\mathcal T}

\newcommand{\PP}{\mathbb P}

\renewcommand{\leq}{\leqslant}
\renewcommand{\geq}{\geqslant}
\renewcommand{\phi}{\varphi}

\title[Representations
of the braid group]
{Quantum Teichm\"uller theory\\
and representations\\ 
of the pure braid group}

\author{Francis Bonahon}
\address {Department
of Mathematics,  University of
Southern California, Los Angeles
CA~90089-2532, U.S.A.}
\email{fbonahon@math.usc.edu}

\date{\today}

\dedicatory{Dedicated to the memory of Xiao-Song Lin,\\
an important mathematician and a great colleague}

\begin{document}

\begin{abstract}
We adapt some of the methods of quantum Teichm\"uller theory to construct a family of representations of the pure braid group of the sphere. 
\end{abstract}

\maketitle

This article presents a variation of the constructions of \cite{BBL, BonLiu} on the finite-dimensional representation theory of the quantum Teichm\"uller space. 

In these two earlier articles, the author and his collaborators considered the quantum Teichm\"uller space  $\T_S^q$  of a connected oriented surface $S$ with negative Euler characteristic, following V. Fock, L. Chekhov \cite{Fok, CF, CF2} and R. Kashaev \cite{Kash3}.  This is designed as a non-commutative deformation of the algebra of rational functions on the classical Teichm\"uller space of $S$, depending on a parameter $q\in \C^*$. It is a purely algebraic object, defined in terms of the combinatorics of the Harer-Penner complex of the ideal cell decompositions of $S$. 

Finite-dimensional representations of the quantum Teichm\"uller space exist only when $q$ is a root of unity, and more precisely if only if a power of $-q$ is equal to $-1$.  Whether one considers irreducible representations as in \cite{BonLiu}, or local representations as in \cite{BBL}, the main results have the same flavor. If $-q$ is a primitive $N$--root of $-1$, a representation of the quantum Teichm\"uller space is classified up to isomorphism  by:
\begin{enumerate}
\item a group homomorphism $r$ from the fundamental group $\pi_1(S)$ to the group $\Isom(\HH^3)$ of orientation-preserving isometries of the hyperbolic space $\HH^3$;
\item the choice, for each peripheral subgroup $\pi$ of $\pi_1(S)$, of a point of the sphere at infinity $ \partial _\infty \HH^3$ which is fixed by $r(\pi)$;
\item the choice of $N$--roots for finitely many complex numbers associated to $r$ and to the fixed point data. 
\end{enumerate}

Not every such data is realized by a representation of the quantum Teichm\"uller space. One needs the group homomorphism and the fixed point data to be \emph{peripherally generic}, in the sense that every ideal triangulation can be realized as the pleating locus of an $r$--invariant pleated surface in $\HH^3$, whose faces are ideal triangles with vertices among the peripheral fixed points; see \cite{BonLiu, BBL} for details. Such a condition is automatically realized if $r$ is injective, which holds in many cases of geometric interest.

As an application, if $G$ is a subgroup of the mapping class group $\pi_0\, \Diff(S)$ which fixes data as in (1-3) above, then $G$ fixes a representation of the quantum Teichm\"uller space up to isomorphism, and the corresponding intertwining isomorphisms provide a projective representation of $G$. This principle is exploited in \cite{BonLiu, BBL} for the case where $G$ is cyclic generated by a pseudo-Anosov isomorphism. From a conceptual point of view, it is probably best understood in terms of the Kashaev bundle constructed in \cite[\S 8]{BBL}. See also \cite{Liu2} for explicit computations. 

The current paper aims at obtaining a similar representation when $G$ is a much larger subgroup of the whole mapping class group $\pi_0 \,\Diff(S)$. For this, we need a homomorphism $r\col \pi_1(S) \to \Isom(\HH^3)$ which is invariant under $G$. The trivial homomorphism is an obvious candidate. Then the fixed point data as in (2) above just consists of the choice of one point of the sphere at infinity $ \partial _\infty \HH^3$ for each puncture of $S$. Respecting this fixed point data will require the elements of $G$ to fix each puncture of $S$. 

A much more serious problem is that  the trivial homomorphism  is very far from being peripherally generic. We consequently need to adapt the definition of the quantum Teichm\"uller space, and restrict attention to ideal triangulations that are realized by the trivial homomorphism. If the fixed points associated to the punctures are all distinct, this just means that each edge of the triangulation joins two distinct punctures, which is no serious constraint as soon as there is more than one puncture. However, one then encounters a new technical problem. We need to know that any two such ideal triangulations can be joined by a sequence of diagonal exchanges, involving only ideal triangulations with the same property. This result may be true in its full generality, but we chose to use a ``cheap trick'' by further restricting the type of ideal triangulations considered. We limit the analysis to the case where $S$ is a punctured sphere, and consider the \emph{Rivin complex}, formed by all cell decompositions of $S$ that can be realized by a convex polyhedron in $\R^3$ inscribed in the sphere. This complex is realized by a cell decomposition of the Teichm\"uller space of $S$ and, in particular, its 1--skeleton is connected, which is exactly the property we need. 

We can now describe the output of this construction. Let $S$ be obtained by removing $r\geq 3$ points from the sphere $\mathbb S^2 \subset \R^3$, so that the mapping class group $\pi_0\, \Diff(S)$ is also the braid group $B_r(\mathbb S^2)$ of the sphere. Let $P_r(\mathbb S^2)$ be the pure braid group, consisting of those mapping classes that fix each puncture. 

Choose a root of unity $q\in \C$  for which $-q$ is a primitive $N$--root of $-1$, as well as an $N$--root of unity $p_j$ for each puncture. 

We then associate to this data a projective representation of dimension $N^{r-3}-1$ for the pure braid group $P_r(\mathbb S^2)$, namely a group homomorphism 
$$
R\col P_r(\mathbb S^2) \to \Aut \left( \mathbb {CP}^{N^{r-3}-1} \right)
$$
valued in the group of linear automorphisms of the projective space $\mathbb{CP}^{N^{r-3}-1} $. This representation depends on the conformal structure of $S$, namely on the specific points $v_1$, $v_2$, \dots, $v_r \in \mathbb S^2$ such that $S = \mathbb S^2 -\{v_1, v_2, \dots, v_r\}$.

In the case of the 4--times punctured sphere, the formulas of \cite{Liu2} provide an easy way to explicitly compute this homomorphism. The computations in the general case are somewhat more elaborated. 

At this point, it is unlikely that much geometric or algebraic insight can be gained from the representations $R$. It was developed as a ``toy model'' for the Volume Conjecture of \cite{Kash2, MM}, which was of great interest to  Xiao-Song Lin towards the end of his life (see \cite{DasLin}, and several unpublished contributions). Many features of our construction are very reminiscent of the challenges that one encounters with this conjecture.  In particular, it would be interesting to let $N$ tend to $\infty$, and determine what happens to the trace functions associated to the corresponding representations $R$. 

We are very grateful to the referee for a careful reading of the original version of this article. 

\section{The Chekhov-Fock algebra of an ideal triangulation}

Let $S$ denote the $r$-times punctured sphere, considered as an abstract topological surface. 

An \emph{ideal cell decomposition} of $S$ is a family $\lambda$ of disjoint arcs properly embedded in $S$ (namely going from puncture to puncture) such that each component of the complement $S-\lambda$ is homeomorphic to an open disk, and is adjacent to at least three sides of arcs of $\lambda$. Equivalently, splitting $S$ open along $\lambda$ provides a finite family of polygons with their vertices removed (namely \emph{ideal polygons}), none of which is a monogon or a digon; these ideal polygons are the \emph{faces} of $\lambda$.  The components of the ideal cell decomposition are also called its \emph{edges}. 

We will identify two ideal cell decompositions of $S$ when they differ only by an isotopy of $S$.

An \emph{ideal triangulation} of $S$ is an ideal cell decomposition whose faces are all ideal triangles.  An Euler characteristic computation shows that an ideal triangulation has $n=3r-6$ sides. 

Fix a non-zero complex number $q = \mathrm e^{\pi\mathrm i \hbar} \in \C^*$.

The \emph{Chekhov-Fock algebra} of an ideal triangulation
$\lambda$ is the algebra $\T_\lambda^q$ over $\C$ defined by generators
$X_1^{\pm1}$, $X_2^{\pm1}$, \dots, $X_n^{\pm1}$  associated to the edges $\lambda_1$, $\lambda_2$, \dots, $\lambda_n$ of $\lambda$, and by
relations
$X_iX_j= q^{2\sigma_{ij}}X_jX_i$ where the integers $\sigma_{ij}\in \{ 0, \pm1, \pm2\}$ are
defined as follows: if $a_{ij}$ is the number of angular sectors  delimited  by $\lambda_i$ and  $\lambda_j$ in the
faces of
$\lambda$, and with $\lambda_i$
coming first counterclockwise, then $\sigma_{ij} = a_{ij}-a_{ji}$. 

In practice, the elements of the Chekhov-Fock algebra $\mathcal
T_\lambda^q$ are Laurent polynomials in the variables $X_i$,
and are manipulated in the usual way except that their multiplication
uses the skew-commutativity relations $X_iX_j= q^{2\sigma{ij}}X_jX_i$.

Given two ideal triangulations $\lambda$ and $\lambda'$ of $S$, one would  like to identify the Chekhov-Fock algebras $\T_\lambda^q$ and $\T_{\lambda'}^q$. This is not quite possible as stated, and one needs to consider the fraction algebras  $\widehat\T_\lambda^q$ and $\widehat\T_{\lambda'}^q$ of  $\T_\lambda^q$ and $\T_{\lambda'}^q$. These fraction algebras exist because the Chekhov-Fock algebras satisfy the so-called Ore Condition; see \cite{Coh}. In practice, the elements of the fraction algebra are rational fractions in the variables $X_i$,
and are manipulated by making use of the skew-commutativity relations $X_iX_j= q^{2\sigma{ij}}X_jX_i$.

\begin{thm}[Chekhov-Fock \cite{Fok, CF, CF2}; see also \cite{Liu1}]
\label{thm:CoordChanges}
As $\lambda$ and $\lambda'$ range over all ideal triangulations of $S$, there exists a family of algebra isomorphisms
$$
\Phi_{\lambda\lambda'}^q \col \widehat  \T_{\lambda'}^q \to \widehat  \T_\lambda^q
$$
such that $\Phi_{\lambda\lambda''}^q = \Phi_{\lambda\lambda'}^q \circ \Phi_{\lambda'\lambda''}^q$ for every ideal triangulations $\lambda$, $\lambda'$, $\lambda''$. 
\end{thm}

The isomorphisms $\Phi_{\lambda\lambda'}^q $ are called \emph{coordinate change isomorphisms} for the quantum Teichm\"uller space of $S$. Indeed, this quantum Teichm\"uller space is the algebra defined by gluing together all the Chekhov-Fock fraction  algebras $\widehat\T_\lambda^q$ by the $\Phi_{\lambda\lambda'}^q $. This turns the $\widehat\T_\lambda^q$ into charts for the quantum Teichm\"uller space, while the isomorphisms $\Phi_{\lambda\lambda'}^q $ play the role of coordinate changes. 

A relatively surprising fact is that these coordinate change isomorphisms $\Phi_{\lambda\lambda'}^q $ are essentially determined by the property of Theorem~\ref{thm:CoordChanges} and by an additional relatively mild condition \cite{Bai}.

\section{Representations of the Chekhov-Fock algebra}

Consider a finite-dimensional representation of the Chekhov-Fock algebra $\T_\lambda^q$, namely an algebra homomorphism $\rho\col \T_\lambda^q \to \End (V)$ valued in the algebra of linear endomorphisms of a finite-dimensional vector space $V$. Elementary considerations on the relation $XY = q^2 YX$ show that these can exist only when $q^2$ is a root of unity. 

We will consequently assume that $q^2$ is a primitive $N$--root of unity for some integer $N\geq 1$. In particular, $q^N = \pm 1$.  Important parts of the representation theory work much better when $q^N = (-1)^{N+1}$, and we will therefore assume that this property holds henceforth.

Since  $q^{2N} = 1$, the relations $X_iX_j= q^{2\sigma{ij}}X_jX_i$ show that each  $X_i^N$ is central in $\T_\lambda^q$. 

The center of $\T_\lambda^q$ contains a few less obvious elements. If $\lambda_{i_1}$, $\lambda_{i_2}$, \dots, $\lambda_{i_k}$ are the edges of $\lambda$ that lead to the $j$--th puncture of $S$, the element 
$$
P_j = q^{-\sum_{u<v} \sigma_{i_u i_v}} X_{i_1} X_{i_2} \dots X_{i_k}
$$
is also central (see \cite[\S 3]{BonLiu}). The $q$--coefficient, known in the physics literature as the Weyl quantum ordering, is here introduced to simplify further computations. In particular, $P_j$ is independent of the ordering of the $X_{i_l}$, and $P_j^N = X_{i_1}^N X_{i_2}^N \dots X_{i_k}^N$. 

Similarly, there is a global central element
$$
H = q^{-\sum_{u<v} \sigma_{uv} }X_1 X_2 \dots X_n,
$$
such that $H^2 = P_1 P_2 \dots P_r$. 

We now restrict attentions to  \emph{irreducible} representations, which leave no proper subspace $W\subset V$ invariant. Then, $\rho$ must send each central element to a scalar multiple of the identity $\Id_V$. In particular, for the above central elements, there exists  scalars $x_1$, $x_2$, \dots, $x_n$, $p_1$, $p_2$, \dots $p_r$, $h \in \C^*$ such that $\rho(X_i^N) = x_i \, \Id_V$, $\rho(P_j) = p_j\, \Id_V$ and $\rho(H) = h\, \Id_V$. 

Note that, if $P_j = q^{-\sum_{u<v} \sigma_{i_u i_v}} X_{i_1} X_{i_2} \dots X_{i_k}$, $p_j^N = x_{i_1} x_{i_2} \dots x_{i_k}$ since $P_j^N = X_{i_1}^N X_{i_2}^N \dots X_{i_k}^N$. Similarly, $h^2 = p_1p_2 \dots p_r$.

\begin{thm}[{\cite[Theorems 20, 21 and 22]{BonLiu}}]
\label{thm:IrredRepClass}
Let the surface $S$ be a sphere minus $r\geq 3$ punctures, and let  $q$ be a primitive $N$--root of $(-1)^{N+1}$. Then, up to isomorphism, an irreducible representation $\rho\col \T_\lambda^q \to \End (V)$ of the Chekhov-Fock algebra of $\T_\lambda^q$ of the ideal triangulation $\lambda$ has dimension $N^{r-3}$ and is classified by the data of: 
\begin{enumerate}
\item for the $i$--th edge $\lambda_i$ of $\lambda$, the number $x_i \in \C^*$ such that  $\rho(X_i^N) = x_i \, \Id_V$;
\item for the $j$--th puncture of $S$ adjacent to the edges $\lambda_{i_1}$, $\lambda_{i_2}$, \dots, $\lambda_{i_k}$ of $\lambda$, the $N$--root $p_j\in \C^*$ of $x_{i_1} x_{i_2} \dots x_{i_k}$ such that $\rho(P_j) = p_j\, \Id_V$;
\item the square root  $h\in \C^*$ of $ p_1p_2 \dots p_r$ such that $\rho(H) = h\, \Id_V$. \qed
\end{enumerate}
\end{thm}

We will refer to the numbers  $x_i$, $p_j$ and $h$ as, respectively,   the \emph{edge weights}, \emph{puncture weights} and \emph{global charge} classifying the representation $\rho$. There is a similar result for surfaces of positive genus, but one then needs additional invariants when $N$ is even. 

Suppose that we are given two ideal triangulations $\lambda$ and $\lambda'$, and a representation $\rho\col \T_\lambda^q \to \End (V)$. We would like to define a representation $\rho\circ \Phi_{\lambda\lambda'}^q \col \T_{\lambda'}^q \to \End (V)$. Unfortunately, the coordinate change isomorphism $\Phi_{\lambda\lambda'}^q$ is valued in the fraction algebra $\widehat\T_\lambda^q$, and dimension considerations easily show that $\rho$ cannot have any extension to $\widehat\T_\lambda^q$. This obliges to use an \emph{ad hoc} definition.

We will say that $\rho\circ \Phi_{\lambda\lambda'}^q \col \T_{\lambda'}^q \to \End (V)$ \emph{makes sense} if, for every $X' \in \T_{\lambda'}^q$, its image $\Phi_{\lambda\lambda'}^q(X') \in \widehat \T_\lambda^q$ can be written as 
$$
\Phi_{\lambda\lambda'}^q(X') = PQ^{-1} = (Q')^{-1} P'
$$ 
for some $P$, $Q$, $P'$, $Q'\in \T_\lambda^q$ such that $\rho(Q)$ and $\rho(Q')$ are invertible in $\End(V)$. When this condition holds, we define $\rho\circ \Phi_{\lambda\lambda'}^q \col \T_{\lambda'}^q \to \End (V)$  by the property that $\rho\circ \Phi_{\lambda\lambda'}^q (X') = \rho(P) \circ \rho(Q)^{-1}$. One easily checks that this is independent of the decomposition $
\Phi_{\lambda\lambda'}^q(X') = PQ^{-1} = (Q')^{-1} P'
$ as above, and that the map $\rho\circ \Phi_{\lambda\lambda'}^q \col \T_{\lambda'}^q \to \End (V)$ so defined is an algebra homomorphism.

\begin{figure}[htb]
\SetLabels
( 0 * 1.03) $v_{\mathrm{left}}$ \\
( 0 * -0.08) $v_-$ \\
( .33 * -0.08) $v_{\mathrm{right}}$ \\
( .33 * 1.03) $v_+$ \\
( .675 * 1.03) $v_{\mathrm{left}}$ \\
( .675 * -0.08) $v_-$ \\
( 1.005 * -0.08) $v_{\mathrm{right}}$ \\
( 1.005 * 1.03) $v_+$ \\
( .13 * .5) $\lambda_1$ \\
( .17 *  1.01) $\lambda_2$ \\
( .34 * .5 ) $\lambda_3$ \\
( .17 * -.08 )  $\lambda_4$ \\
( -.02 * .5 ) $\lambda_5$ \\
(  .87* .5 ) $\lambda_1'$ \\
( .83 * 1.01 ) $\lambda_2'$ \\
( 1.02 * .5 ) $\lambda_3'$ \\
( .83 * -.08 ) $\lambda_4'$ \\
( .66 * .5 ) $\lambda_5'$ \\
\endSetLabels
\centerline{
\AffixLabels{\includegraphics{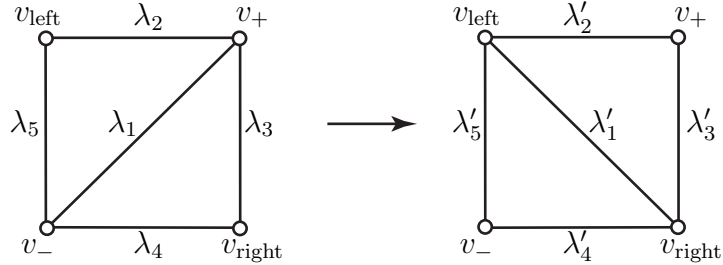}}}
\caption{A diagonal exchange}
\label{fig:DiagEx}
\end{figure}

In general, given a representation  $\rho\col \T_\lambda^q \to \End (V)$  and another ideal triangulation $\lambda'$, it is not easy to decide when the representation $\rho\circ \Phi_{\lambda\lambda'}^q \col \T_{\lambda'}^q \to \End (V)$ makes sense. This is relatively simple when $\lambda$ and $\lambda'$ differ only by one edge, as in Figure~\ref{fig:DiagEx}. We then say that $\lambda$ and $\lambda'$ differ by a \emph{diagonal exchange}. The diagonal exchange is \emph{embedded} when the four sides of the square where the diagonal exchange takes place are distinct, namely when there are no outside gluings between these sides.

\begin{prop}[{\cite[Lemmas 27 and 29]{BonLiu}}]
\label{prop:DiagExWeights}
Let the ideal triangulations $\lambda$ and $\lambda'$ differ by an embedded diagonal exchange as in Figure~\ref{fig:DiagEx}, with the edge indexing shown there. Suppose that the irreducible representation $\rho\col \T_\lambda^q \to \End (V)$ is classified by the edge weights $x_i$, the puncture weights $p_j$ and the global charge $h$. If $x_i \not = -1$, then the representation $\rho \circ \Phi_{\lambda\lambda'}^q \col \T_{\lambda'}^q \to \End(V)$ is well-defined, and is classified by the same punctured weights $p_j$, the same global charge $h$, and by the edge weights $x_i'$ defined by
\begin{align*}
x'_1 &= x_1 ^{-1}\\
x_2' &= (1+x_1) x_2\\
x_3' &= (1 + x_1^{-1})^{-1} x_3\\
x_4' &= (1+x_1) x_4\\
x_5' &= (1 + x_1^{-1})^{-1} x_5\\
x_i' &= x_i \text{ if } i \not= 1, 2, 3, 4, 5.
\end{align*}
\vskip -\belowdisplayskip 
\vskip - \baselineskip \qed
\end{prop}

This is a relatively simple consequence of the formula for the coordinate change isomorphism $\Phi_{\lambda\lambda'}^q$ for a diagonal exchange. With significantly more work to control the invertibility of the images of denominators, it can be extended to the case where $\lambda$ and $\lambda'$ can be connected by a sequence of diagonal exchanges.

\begin{prop}[{\cite[Lemmas 25 and 26]{BonLiu}}]
\label{prop:ManyDiagEx}
Consider a sequence of ideal triangulations $ \lambda^{(0)}$, $ \lambda^{(1)}$, \dots, $ \lambda^{(k)}$, where each $\lambda^{(i+1)}$ is obtained from $\lambda^{(i)}$ by a diagonal exchange. Suppose in addition that we are given  irreducible representations $\rho_{\lambda^{(i)}} \col \T_{\lambda^{(i)}} \to \End(V)$ such that $\rho_{\lambda^{(i+1)}}= \rho_{\lambda^{(i)}} \circ \Phi_{\lambda^{(i)} \lambda^{(i+1)}}^q$ for every $i$. Then, $\rho_{\lambda^{(k)}}= \rho_{\lambda^{(0)}} \circ \Phi_{\lambda^{(0)} \lambda^{(k)}}^q$. \qed
\end{prop}

\section{Geometric data}
\label{sect:CrossRatios}

From now on, we will identify $S$ with the complement $\widehat \C -\{ v_1, v_2, \dots, v_r \}$  of $r$ distinct points $v_1$, $v_2$, \dots $v_r$ in the Riemann sphere $\widehat \C = \C \cup \{ \infty \}$. Namely, we choose a conformal structure on $S$.

An ideal triangulation $\lambda$ is \emph{simple} if every edge of $\lambda$ joins two distinct punctures of $S$. 

The conformal structure of $S$ then defines a system on complex edge weights on $\lambda$ defined as follows. Choose an arbitrary orientation for the edge $\lambda_i$, so that it goes from the puncture $v_-$ to the puncture $v_+\in \widehat \C$. Let $T_{\mathrm{left}}$ and $T_{\mathrm{right}}$ be the two faces of $\lambda$ containing $\lambda_i$ and respectively located to the left and to the right of $\lambda_i$ for the orientation specified, and let $v_{\mathrm{left}}$ and $v_{\mathrm{right}}$ be the vertices of these two triangles that are different from $v_-$ and $v_+$. See  Figure~\ref{fig:DiagEx}. Then, the weight $x_i \in \C^*$ is defined as the cross-ratio
$$
x_i = -\frac{(v_+ - v_{\mathrm{left}})(v_- - v_{\mathrm{right}})}{(v_+ - v_{\mathrm{right}})(v_- - v_{\mathrm{left}})}.
$$

In the case of a diagonal exchange, an elementary computation provides:

\begin{lem}
\label{lem:DiagExCrossRatio}
Let the simple ideal triangulations $\lambda$ and $\lambda'$ differ by an embedded diagonal exchange as in Figure~\ref{fig:DiagEx}, with the edge indexing shown there. Then the above edge weights satisfy:
\begin{align*}
x'_1 &= x_1 ^{-1}\not =-1\\
x_2' &= (1+x_1) x_2\\
x_3' &= (1 + x_1^{-1})^{-1} x_3\\
x_4' &= (1+x_1) x_4\\
x_5' &= (1 + x_1^{-1})^{-1} x_5.\\
x_i' &= x_i \text{ if } i \not= 1, 2, 3, 4, 5.
\end{align*}
\vskip -\belowdisplayskip 
\vskip - \baselineskip \qed
\end{lem}
Note the similarity with the formulas of Proposition~\ref{prop:DiagExWeights}. 

In the same way, around a puncture:
\begin{lem}
\label{lem:CrossRatioPuncture}
If $\lambda_{i_1}$, $\lambda_{i_2}$, \dots, $\lambda_{i_k}$ are the edges adjacent to the $j$--th puncture of $S$, then the above cross-ratios are such that 
$$
x_{i_1} x_{i_2} \dots x_{i_k} =1.
$$
\vskip - \belowdisplayskip 
\vskip -\baselineskip \qed
\end{lem}

We can use these edge weights $x_i$ to construct representations of the Chekhov-Fock algebras attached to simple ideal triangulations of the punctured sphere $S = \widehat \C -\{ v_1, v_2, \dots, v_r \}$.

For each puncture $v_i$, choose an $N$--root of unity $p_j = q^{2n_j}$. Then choose a square root $h=\sqrt[2]{p_1p_2 \dots p_r}$. If $\lambda$ is a simple ideal triangulation of $S$, Theorem~\ref{thm:IrredRepClass} and Lemma~\ref{lem:CrossRatioPuncture} then provide an irreducible representation $\rho\col \T_\lambda^q \to \End(V)$, associated to the cross-ratio edge weights $x_i$, the $N$--roots $p_j$ and the square root $h$. 

If we combine Proposition~\ref{prop:DiagExWeights}, Proposition~\ref{prop:ManyDiagEx} and Lemma~\ref{lem:DiagExCrossRatio}, we then obtain:

\begin{lem}
\label{lem:ManySimpleDiagEx}
Let $\rho\col \T_\lambda^q \to \End(V)$ and $\rho'\col \T_{\lambda'}^q \to \End(V')$ be irreducible representations associated to the simple ideal triangulations $\lambda$ and $\lambda'$ as above. Suppose that $\lambda$ and $\lambda'$ can be connected by a sequence of simple ideal triangulations $ \lambda^{(0)}$, $ \lambda^{(1)}$, \dots, $ \lambda^{(k)}$, where each $\lambda^{(i+1)}$ is obtained from $\lambda^{(i)}$ by an embedded diagonal exchange. Then the representation $\rho \circ \Phi_{\lambda\lambda'}^q$ makes sense, and is isomorphic to $\rho'$. \qed

\end{lem}

\section{Rivin triangulations}

In order to apply Lemma~\ref{lem:ManySimpleDiagEx}, we will restrict the type of ideal triangulations that we consider.

An example of an ideal cell decomposition is provided by a convex polyhedron $C$ with $r$ vertices in $\R^3$. If we choose a homeomorphism between $S$ and the complement of the vertices in the boundary $\partial C$,  the pre-images of the edges of  $C$ form an ideal cell decomposition of $S$. 
A famous theorem of Steinitz (see for instance \cite{Grunbaum}) asserts that an ideal cell decomposition is obtained in this way if and only if each edge of $\lambda$ joins different punctures, if no two distinct edges join the same punctures,  and if there is no curve in $S$ which cuts $\lambda$ in exactly two points contained in different edges.

An ideal cell decomposition $\lambda$ of $S$ satisfies the \emph{Rivin Condition}, or is a \emph{Rivin cell decomposition}, if it comes from a convex polyhedron $C$ which, in addition, is inscribed in the sphere $\mathbb S^2$. More precisely, $\lambda$ satisfies the Rivin condition if there exists $v_1$, $v_2$, \dots, $v_r \in \mathbb S^2$ and a homeomorphism between $S$ and $\partial C- \{ v_1, v_2, \dots v_r\}$, where $C$ is the convex hull of $ \{ v_1, v_2, \dots v_r\}$ in $\R^3$, for which $\lambda$ corresponds to the edges of the polyhedron $C$. We will allow the case where the points $v_1$, $v_2$, \dots, $v_r$ are coplanar, in which case the convex hull $C$ is a flat polygon, and $\lambda$ consists of a chain of $r$ edges subdividing $S$ into two ideal $r$--gons. 

From our point of view, a fundamental property of Rivin cell decompositions is that they are simple, in the sense that each edge joins two distinct punctures of $S$.

Igor Rivin \cite{Riv1, Riv2, Riv3, Riv4} (see also \cite{Gueritaud}) showed that the Rivin condition is equivalent to the existence of a system of real weights on the edges of $\lambda$ satisfying a certain number of linear equalities and inequalities. In particular, for a given ideal cell decomposition, the validity of this property can be decided by a linear programming algorithm. The simplest example of an ideal triangulation that does not satisfies the Rivin Condition is provided by the convex polyhedron obtained by gluing a triangular pyramid with low height on each of the faces of a tetrahedron. The corresponding ideal triangulation of the 6--times punctured sphere is by construction realized by a convex polyhedron, but can be shown to be irrealizable by a convex polyhedron inscribed in the sphere.

\begin{lem}
\label{lem:DiagEx}
If $r\geq 5$, any two Rivin triangulations $\lambda$, $\lambda'$ of the $r$--times punctured sphere can be connected by a sequence of Rivin triangulations $\lambda = \lambda^{(0)}$, $ \lambda^{(1)}$, \dots, $ \lambda^{(l)}= \lambda'$, where any $\lambda^{(i)}$ is obtained from $\lambda^{(i-1)}$ by an embedded diagonal exchange. 
\end{lem}

\begin{proof}
Consider the (cusped) Teichm\"uller space $\T(S)$. An element of $\T(S)$ consists of a homeomorphism $f\col S \to \mathbb S^2 - \{ v_1, v_2, \dots, v_r\}$ between $S$ and the complement of $r$ points in the unit sphere $\mathbb S^3$, considered up to precomposition by an isotopy of $S$ and up to postcomposition by a conformation transformation of $\mathbb S^2$.

If $[f]\in \T(S)$ is represented by the homeomorphism $f\col S \to \mathbb S^2 - \{ v_1, v_2, \dots, v_r\}$, we can consider the convex hull $C$ of the points $v_1$, $v_2$, \dots, $v_r$ in $\R^3$. The closest point projection can be perturbed to a homeomorphism $\mathbb S^2 \to \partial C$, well-defined up to isotopy fixing $\{v_1, v_2, \dots, v_r\}$. For this identification, the edges and faces of $C$ provide an ideal cell decomposition of  $\mathbb S^2 - \{ v_1, v_2, \dots, v_r\}$, which can be pulled back to an ideal cell decomposition $\lambda_{[f]}$ of $S$ by the homeomorphism $f$. 

As before, when the $r$ points $v_1$, $v_2$, \dots, $v_r$ are coplanar, $\lambda_{[f]}$ consists of $r$ edges dividing $S$ into two ideal $r$--gons. 

Every conformal transformation of $\mathbb S^2$ has a unique extension to a projective transformation of the ball $\mathbb B^3$ bounded by $\mathbb S^2$, as in the projective model for the hyperbolic 3--space. Since this projective extension sends convex hull to convex hull, it follows that                                                                                                                                                                                                                                                                                                                                                                                                                                                                                                                                                                                                                                                   the ideal cell decomposition $\lambda_{[f]}$ of $S$ depends only on the class $[f]\in \T(S)$, up to isotopy. 

Let $\lambda$ and $\lambda'$ be two Rivin triangulations. By definition, there exists $[f]$, $[f']\in\T(S)$ such that $\lambda_{[f]} = \lambda$ and $\lambda_{f']}=\lambda'$. Since the Teichm\"uller space $\T(S)$ is connected, we can join these two elements by a path $t\mapsto [f_t]\in \T(S)$, $0\leq t\leq 1$, where the homeomorphism $f_t \col S \to \mathbb S^2 - \{v_1^{(t)}, v_2^{(t)}, \dots, v_r^{(t)} \}$ representing $[f_t]$ is such that $f_0=f$ and $f_1=f'$. By general position, we can arrange that the map $t \mapsto (v_1^{(t)}, v_2^{(t)}, \dots, v_r^{(t)})\in (\mathbb S^2)^r$ is transverse to the codimension 1 stratum where exactly four of these vertices $v_i^{(t)}$  are coplanar. 

Outside of the values of $t$ where four vertices become coplanar in this way, the ideal cell decomposition $\lambda_{[f_{t}]}$ is an ideal triangulation, and is independent of $t$ under small perturbation. 

As one crosses this codimension 1 stratum at $t=t_0$,  four vertices become coplanar, but the convex hull of the $v_i^{(t)}$ remain 3--dimensional by our hypothesis that $r\geq 5$. If the four coplanar vertices are not contained in the same face of $\lambda_{[f_{t_0}]}$, the ideal triangulation $\lambda_{[f_{t}]}$ remains unchanged as $t$ crosses $t_0$. Otherwise, $\lambda_{[f_{t}]}$ changes by an embedded diagonal exchange in the square face containing these vertices. 
\end{proof}

When $r=4$, any Rivin triangulation is combinatorially equivalent to a tetrahedron, with three edges arriving at each puncture. In this case, the  convex hull becomes flat as one crosses the codimension 1 sratum. The same proof then shows that any two Rivin triangulations of the 4--times punctured sphere can be joined by a sequence of  \emph{double diagonal exchanges} illustrated in Figure~\ref{fig:DoubleDiagEx}.

\begin{figure}[htbp]
\includegraphics{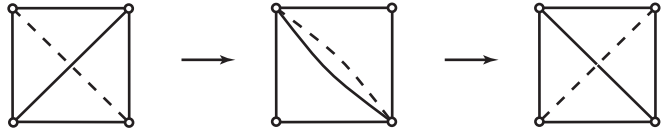}
\caption{A double diagonal exchange}
\label{fig:DoubleDiagEx}
\end{figure}

Note that in a double diagonal exchange, the intermediate ideal triangulation is not a Rivin triangulation, but is nevertheless simple. Therefore:

\begin{lem}
\label{lem:DbleDiagEx}
Any two Rivin triangulations of the $4$--puncture sphere can be joined to each other by a sequence of embedded diagonal exchanges involving only simple ideal triangulations.  \qed
\end{lem}

Note that the case where $r=3$ is particularly trivial, since in this case $S$ admits only one Rivin triangulation, decomposing $S$ into two triangles.

\section{Constructing a representation of $P_r(\mathbb S^2)$}

We can now put everything together. We fix a conformal structure on the punctured sphere $S$, namely an identification between $S$ and the complement $\widehat \C -\{ v_1, v_2, \dots, v_r \}$ of $r$ points in the Riemann sphere $\widehat\C$. 

After this geometric data, let us choose some algebraic (or number-theoretic) data. We take a root of unity $q$ such that $-q$ is a primitive $N$--root of $(-1)$. Choose an $N$--root of unity $p_j = q^{2n_j}$ for each puncture of $S$, and a square root $h = \sqrt[2]{p_1 p_2 \dots p_r}$. 

Now, pick a Rivin triangulation $\lambda$ of $S$. The points $v_i$ associate a cross-ratio weight $x_i \in \C^*$ to each edge of $\lambda$. As in \S\ref{sect:CrossRatios}, the edge weights $x_i$, the $N$--roots $p_j$ and the square root $h$ define, up to isomorphism, an irreducible  representation   $\rho_\lambda\col \T_\lambda^q \to \End (V)$ of the Chekhov-Fock algebra of $\lambda$.

Consider a diffeomorphism $\phi\col S \to S$, and the ideal triangulation $\phi(\lambda)$. Combining Lemmas~\ref{lem:ManySimpleDiagEx}, \ref{lem:DiagEx} and \ref{lem:DbleDiagEx}, the representation  $\rho \circ \Phi_{\lambda\phi(\lambda)}^q \col \T_{\phi(\lambda)}^q \to \End(V)$ makes sense. By Proposition~\ref{prop:DiagExWeights} and Lemma~\ref{lem:DiagExCrossRatio}, this  representation $\rho \circ \Phi_{\lambda\phi(\lambda)}^q$ is classified by the cross-ratio edge weights associated to the edges of $\phi(\lambda)$ by the points $v_1$, $v_2$, \dots $v_r \in \widehat \C$, plus the same puncture weights $p_j$ and the same global charge $h$ as $\rho$. 

There is also a simpler way to obtain a representation of $\T_{\phi(\lambda)}^q$ from $\rho$. Indeed, $\phi$ provides a natural isomorphism $I_{\lambda, \phi} \col \T_{\phi(\lambda)}^q \to \T_\lambda^q$, which sends the generator associated to an edge of $\phi(\lambda)$ to the generator associated to the corresponding edge of $\lambda$. The representation  $\rho \circ I_{\lambda,\phi} \col \T_{\phi(\lambda)}^q \to \End(V)$ is classified by assigning to each edge of $\phi(\lambda)$ the cross-ratio weight of the corresponding edge of $\lambda$,  by assigning to each puncture the puncture weight $p_j$ associated by $\rho$ to the preimage of this puncture under $\phi$, and by the same  global charge $h$ as $\rho$. 

Now, suppose in addition that $\phi$ fixes every puncture of $S$. Then each edge of $\lambda$ has the same cross-ratio weight as its image under $\phi$, since the corresponding cross-ratios involve the same points. Also, $\rho \circ \Phi_{\lambda\phi(\lambda)}^q $ and $\rho \circ I_{\lambda,\phi} $ then assign the same weight $p_j$ to each puncture, and have the same global charge $h$.  It follows that the representations $\rho \circ \Phi_{\lambda\phi(\lambda)}^q $ and $\rho \circ I_{\lambda,\phi} $ are isomorphic.

This means that there is an isomorphism $L_\phi \col V \to V$ such that 
$$ \rho \circ \Phi_{\lambda\phi(\lambda)}^q  =   L_\phi \circ \left( \rho \circ I_{\lambda,\phi}  (X) \right) \circ L_\phi^{-1}
$$
for every $X\in \T_{\phi(\lambda)}^q$. 

It may be more convenient to rewrite this property with the following definition. If $L\col W \to W'$ is a linear isomorphism between vector spaces, its \emph{adjoint homomorphism} $\Ad_L\col \End(W) \to \End(W')$ is defined by the property that $\Ad_L(A) = L \circ A \circ L^{-1}$. Then the above property can be rewritten as
$$
 \rho \circ \Phi_{\lambda\phi(\lambda)}^q = \Ad_{L_\phi} \circ \rho \circ I_{\lambda,\phi}  .
$$

The isomorphism $L_\phi$ is not unique since multiplying it by a constant does not change the adjoint homomorphism $\Ad_{L_\phi}$.  However, because the representation $\rho \circ I_{\lambda,\phi}$ is irreducible, this is the only ambiguity and $L_\phi$ is unique up to rescaling. In particular, the projective isomorphism $\PP L_\phi \col \PP(V) \to \PP(V)$ of the projective space of $V$ is uniquely determined. 

Recall from Theorem~\ref{thm:IrredRepClass} that the dimension of $V$ is equal to $ N^{r-3} $. We can consequently identity $\PP(V)$ to the complex projective space $\mathbb {CP}^{N^{r-3}-1}$. 

\begin{thm}
The map
$$
R \col P_r(\mathbb S^2) \to \Aut\bigl (\mathbb {CP}^{N^{r-3}-1} \bigr)
$$
which associates $\PP L_\phi$ to $[\phi] \in P_n (\mathbb S^2)$, is a group homomorphism.  Up to conjugation by an element of $\Aut(\mathbb {CP}^{N^{r-3}-1})$, it is independent of the  choices of the ideal triangulation $\lambda$ and of the representation $\rho$. 
\end{thm}

\begin{proof}
We first prove the independence up to conjugation. Let us assume that we start with another Rivin triangulation $\lambda'$ and a representation $\rho'\col \T_{\lambda'}^q \to \End(V')$ classified by the cross-ratio edge weights associated to the edges of $\lambda'$, the puncture weights $p_j$ and the global charge $h$. Then, Lemmas~\ref{lem:ManySimpleDiagEx}, \ref{lem:DiagEx},  \ref{lem:DbleDiagEx} and Proposition~\ref{prop:DiagExWeights} show that $\rho\circ \Phi_{\lambda\lambda'}^q$ makes sense and is isomorphic to $\rho'$ by an isomorphism $L_{\lambda\lambda'}\col V' \to V$. Namely, 
$$
\rho\circ \Phi_{\lambda\lambda'}^q = \Ad_{L_{\lambda\lambda'}} \circ \rho'. 
$$
As a consequence, 
$$
\rho'\circ \Phi_{\lambda'\lambda}^q = \Ad_{L_{\lambda\lambda'}}^{-1} \circ \rho = \Ad_{L_{\lambda\lambda'}^{-1}} \circ \rho. 
$$

We will use the critical property that the coordinate change isomorphisms $\Phi_{\lambda\lambda'}^q$ are natural with respect to the action of the diffeomorphism group of $S$. This means that, for every diffeomorphism $\psi\col S \to S$, 
$$
\Phi_{\psi(\lambda)\psi(\lambda')}^q = I_{\lambda, \psi}^{-1} \circ \Phi_{\lambda\lambda'}^q \circ I_{\lambda', \psi}
$$
where $I_{\lambda, \psi} \col \T_{\psi(\lambda)}^q \to \T_{\lambda}^q $ and $I_{\lambda', \psi} \col \T_{\psi(\lambda')}^q \to \T_{\lambda'}^q $ are the canonical isomorphisms induced by $\psi$. This property is an immediate consequence of the construction of the coordinate change isomorphisms $\Phi_{\lambda\lambda'}^q$ in \cite{CF,CF2,Liu1}. 

If we replace $\rho\col \T_{\lambda}^q \to \End(V)$ by $\rho'\col \T_{\lambda'}^q \to \End(V')$ in the above construction, $L_\phi\col V \to V$ gets replaced by an isomorphism $L_\phi' \col V' \to V'$ such that $\Ad_{L_\phi'} \circ \rho' \circ I_{\lambda', \phi} =   \rho' \circ \Phi_{\lambda'\phi(\lambda')}^q $. Then,
\begin{align*}
\Ad_{L_\phi'} \circ \rho'  &=   \rho' \circ \Phi_{\lambda'\phi(\lambda')}^q \circ I_{\lambda', \phi}^{-1}\\
&=  \rho' \circ \Phi_{\lambda'\lambda}^q \circ \Phi_{\lambda\phi(\lambda)}^q \circ \Phi_{\phi(\lambda)\phi(\lambda')}^q  \circ I_{\lambda', \phi}^{-1}\\
&= \Ad_{L_{\lambda\lambda'}^{-1}} \circ \rho\circ  \Phi_{\lambda\phi(\lambda)}^q \circ I_{\lambda, \phi}^{-1} \circ\Phi_{\lambda\lambda'}^q\\
 &= \Ad_{L_{\lambda\lambda'}^{-1}}\circ  \Ad_{L_\phi} \circ \rho \circ\Phi_{\lambda\lambda'}^q\\
&= \Ad_{L_{\lambda\lambda'}^{-1}} \circ  \Ad_{L_\phi} \circ  \Ad_{L_{\lambda\lambda'}} \circ \rho'\\
&= \Ad_{L_{\lambda\lambda'}^{-1} \circ  L_\phi \circ  L_{\lambda\lambda'}} \circ \rho'.
\end{align*}
By irreducibility of $\rho'$, we conclude that $L_\phi'$ is equal to $L_{\lambda\lambda'}^{-1} \circ  L_\phi \circ  L_{\lambda\lambda'}$ up to multiplication by a constant, so that $ \PP L_\phi'= \left( \PP L_{\lambda\lambda'} \right )^{-1} \circ  \PP L_\phi \circ  \PP L_{\lambda\lambda'}$. 

Since this holds for every $\phi$, we see that replacing $\lambda$ and $\rho$ by $\lambda'$ and $\rho'$ in the construction only changes the map $R\col P_r(\mathbb S^2) \to \Aut\bigl (\PP( V )\bigr)$ by conjugation by $ \PP L_{\lambda\lambda'}$.

Having proved this, we now consider the composition of two diffeomorphisms $\phi$ and $\psi \col S \to S$, each fixing the punctures of $S$. We will first apply the above computation to the case where $\lambda' = \phi(\lambda)$, $\rho' = \rho \circ I_{\lambda, \phi}$ and $\phi$ is replaced with $\psi$. By definition of $L_\phi$, 
$$
\rho \circ \Phi_{\lambda\lambda'} ^q = \Ad_{L_\phi} \circ \rho',
$$
so that we can take $L_{\lambda\lambda'} = L_\phi$ in the above argument, which gives  that
$$
\rho' \circ \Phi_{\lambda'\psi(\lambda')}^q = \Ad_{L_\phi^{-1} \circ L_\psi \circ L_\phi}\circ  \rho' \circ I_{\lambda', \psi}.
$$
Then,
\begin{align*}
\rho \circ \Phi_{\lambda\,  \psi\circ \phi(\lambda)} ^q &= \rho\circ  \Phi_{\lambda\lambda'}^q \circ \Phi_{\lambda' \psi(\lambda') }^q \\
&= \Ad_{L_\phi} \circ \rho' \circ \Phi_{\lambda' \psi(\lambda') }^q \\
&= \Ad_{L_\phi} \circ \Ad_{L_\phi^{-1} \circ L_\psi \circ L_\phi} \circ \rho' \circ I_{\lambda', \psi} \\
&= \Ad_{L_\psi \circ L_\phi} \circ \rho  \circ I_{\lambda, \phi}  \circ I_{\phi(\lambda), \psi}  \\
&= \Ad_{L_\psi \circ L_\phi} \circ \rho  \circ I_{\lambda, \psi \circ \phi} 
\end{align*}

This shows that $L_{\psi \circ \phi}$ is equal to $L_\psi \circ L_\phi$ up to scalar multiplication. Therefore, $\PP L_{\psi \circ \phi} = \PP L_\psi \circ \PP L_\phi$, which proves that $R$ is a group homomorphism. 
\end{proof}

\end{document}